\documentclass[a4paper,12pt, times]{amsart}
\usepackage[utf8x]{inputenc}

\usepackage{amsmath}
\usepackage{amsfonts}
\usepackage{amssymb}
\usepackage[mathscr]{eucal}
\usepackage{enumerate}
\usepackage{hyperref}
\usepackage{times}
\usepackage[varg]{txfonts}

\newtheorem{thm}{Theorem}[section]
\newtheorem{lemma}{Lemma}[section]
\newtheorem{cor}{Corollary}[section]
\newtheorem{defin}{Definition}
\newtheorem{ex}{Example}
\newtheorem{rem}{Remark}[section]
\numberwithin{equation}{section}

\usepackage[top=3cm, bottom=3cm, left=2cm, right=2cm]{geometry}

\begin{document}
\title{The Luk\'{a}cs--Olkin--Rubin theorem on symmetric cones}

\author{Eszter Gselmann}

\dedicatory{Dedicated to Professor K\'{a}roly Lajk\'{o} on his seventieth birthday. }



\thanks{This research has been supported by the Hungarian Scientific Research Fund (OTKA)
Grant NK 814 02 and by the 'Lend\"{u}let' Program (LP2012-46/2012) of the Hungarian Academy of Sciences.}

\maketitle

\begin{abstract}
In this paper we prove a Luk\'{a}cs type characterization theorem 
of the Wishart distribution on 
Euclidean simple Jordan algebras under weak and natural regularity assumptions on 
the densities. 
\end{abstract}

\section{Introduction and preliminaries}

The main purpose of this paper is to prove a generalization of the most celebrated result 
in the field of characterization problems concerning probability distribution, the theorem 
of Luk\'{a}cs. This is contained in the following, see also Luk\'{a}cs \cite{Luk55}. 

\begin{thm}[Luk\'{a}cs]
If $X$ and $Y$ are non-degenerate, independent random variables, then the random variables
\[
    V=X+Y \quad \text{ and }\quad U = \frac{X}{X+Y} 
\]
are independently distributed 
if and only if both $X$ and $Y$ have gamma distributions with the same scale parameter.
\end{thm}

Since its 1955 appearance this theorem has been 
extended and generalized in several ways. 
For example, in Olkin--Rubin \cite{OlkRub62} and also in Casalis--Letac \cite{CasLet96} the authors
extended the above result to matrix and symmetric cone variable distributions, respectively. 
The main problem here is that in such  situations there is no unique way of defining the quotient. 
Therefore, the authors quoted above initiated the notion of a so-called division algorithm and under some strong assumption 
they proved an analogue of Luk\'{a}cs's result. To avoid this assumption, in Bobecka--Weso{\l}owski 
\cite{BobWes02} it was assumed that the densities are strictly positive and they are twice 
differentiable. This regularity assumption was weakened in Ko{\l}odziejek \cite{Kol12}. In that paper 
only continuity was assumed.

Therefore the main aim of this paper is to give a unified proof of analogue of Luk\'{a}cs's theorem 
in the symmetric cone setting, assuming only the existence of densities. 

\subsection{Euclidean Jordan algebras}

In this subsection we collected some definitions and statements from the theory of 
Euclidean Jordan algebras that will be used subsequently. For further 
results we refer to the monograph of Faraut and Kor\'{a}nyi, see \cite[Section III. 1.]{FarKor94}.

\begin{defin}
 We say that $(\mathbb{E}; \langle \cdot, \cdot \rangle; \cdot)$ is a \emph{Euclidean Jordan algebra}, if 
\begin{enumerate}[(i)]
 \item $(\mathbb{E}; \langle \cdot, \cdot \rangle)$ is an inner product space 
\item there exists $e\in\mathbb{E}$ such that 
\[
 xe=ex=x
\]
is fulfilled for any $x\in V$
\item for all $x, y, z\in\mathbb{E}$ 
\begin{enumerate}[(a)]
 \item $xy=yx$
\item $x(x^{2}y)=x^{2}(xy)$
\item $\langle x, yz\rangle =\langle xy, z\rangle$
\end{enumerate}
holds. 
\end{enumerate}
\end{defin}

\begin{defin}
 Let $(\mathbb{E}; \langle \cdot, \cdot \rangle; \cdot)$ be a Euclidean Jordan algebra, fix an $x\in\mathbb{E}$ and let us consider the maps 
$\mathbb{P}(x)$ and $\mathbb{L}(x)$ defined by 
\[
 \mathbb{L}(x)y=xy
\]
and 
\[
 \mathbb{P}(x)=2\mathbb{L}^{2}(x)-\mathbb{L}(x^{2}). 
\]
Then the mapping 
\[
 \mathbb{E}\ni x \longmapsto \mathbb{P}(x)\in\mathrm{End}(\mathbb{E})
\]
is called the \emph{quadratic representation} of $\mathbb{E}$. 

An element $z\in\mathbb{E}$ is said to be \emph{invertible} if there exists 
$w\in\mathbb{E}$ such that 
\[
 \mathbb{L}(z)w=e. 
\]
Then $w$ is called the \emph{inverse of $z$} and it will be denoted by $z^{-1}$. 
\end{defin}

\begin{defin}
 The Euclidean Jordan algebra $(\mathbb{E}; \langle \cdot, \cdot \rangle; \cdot)$ is 
said to be \emph{simple} if it is not a Cartesian product of two nontrivial Euclidean Jordan algebras. 
\end{defin}

\begin{ex}
Let $\mathbb{K}=\left\{\mathbb{R}, \mathbb{C}, \mathbb{H}, \mathbb{O}\right\}$ and 
let denote $\mathbf{S}_{r}(\mathbb{K})$ the set of all $r\times r$ Hermitian matrices with entries in 
$\mathbb{K}$. 
If $x, y\in\mathbf{S}_{r}(\mathbb{K})$ let 
\[
 \langle x, y\rangle =\mathrm{Tr}(x\overline{y})
\]
and 
\[
 \mathbb{L}(x)y=\dfrac{xy+yx}{2}. 
\]
Then $\left(\mathbf{S}_{r}(\mathbb{K}); \langle \cdot, \cdot \rangle; \cdot\right)$ is a 
Euclidean Jordan algebra. 
Furthermore, in this case
\[
 \mathbb{P}(y)x=yxy 
\qquad 
\left(x, y\in \mathbb{E}\right). 
\]
\end{ex}

\begin{ex}
Let $n\geq 2$ and let us consider $\mathbb{R}^{n+1}$ with the 
usual inner product and with the Jordan product 
\[
 (x_{0}, x_{1}, \ldots, x_{n}) \cdot 
(y_{0}, y_{1}, \ldots, y_{n})=
\left(\sum_{i=0}^{n}x_{i}y_{i}, x_{0}y_{1}+x_{1}y_{0}, \ldots, x_{0}y_{n}+x_{n}y_{0}\right). 
\]
Then $\mathbb{R}^{n+1}$ is a Euclidean Jordan algebra. 
\end{ex}

The following result can also be found in Faraut--Kor\'{a}nyi, see \cite[Theorem V.3.7]{FarKor94}. 

\begin{thm}
 Up to linear isomorphism there exist only the following Euclidean simple Jordan algebras
\begin{enumerate}[(i)]
 \item $\mathbf{S}_{r}(\mathbb{R})$, where $r\geq 1$ is arbitrary;
\item $\mathbf{S}_{r}(\mathbb{C})$, where $r\geq 2$ is arbitrary;
\item $\mathbf{S}_{r}(\mathbb{H})$, where $r\geq 2$ is arbitrary;
\item $\mathbf{S}_{3}(\mathbb{O})$;
\item $\mathbb{R}^{n+1}$ appearing in the previous example. 
\end{enumerate}
\end{thm}

\begin{thm}[{\cite[Section III.2]{FarKor94}}]
 Let $\mathbb{E}$ be a Euclidean Jordan algebra. Then the set 
\[
 \mathscr{V}=\left\{x^{2}\, \vert \, x \text{ is invertible}\right\}
\]
is and open convex cone that has the following properties. 
\begin{enumerate}[(i)]
 \item The closure of $\mathscr{V}$ is the set 
\[
 \overline{\mathscr{V}}=\left\{x^{2}\, \vert \, x\in\mathbb{E}\right\}. 
\]
\item If $y\in\mathbb{E}$ and for all $x\in\overline{\mathscr{V}}\setminus \left\{0 \right\}$
\[
 \langle y, x\rangle >0
\]
holds, then $y\in \mathscr{V}$. 
\item Let $y\in\mathscr{V}$, then for all $x\in\overline{\mathscr{V}}\setminus \left\{0 \right\}$ 
we have 
\[
\langle y, x\rangle >0.  
\]
\item If $x\in\mathbb{E}$ is invertible, then 
\[
 \mathbb{P}(x)(\mathscr{V})=\mathscr{V}. 
\]
\item For all $x\in\mathscr{V}$ $\mathbb{P}(x)>0$. 
\item The cone $\mathscr{V}$ is the connected component of the identity in the 
set of invertible elements. The cone $\mathscr{V}$ is the set of all elements 
$x\in\mathbb{E}$ for which $\mathbb{L}(x)$ is positive definite. 
\end{enumerate}
\end{thm}

Since each simple Jordan algebra corresponds to a symmetric cone, up 
to a linear isomorphism there exist  also only five types of symmetric cones. 
We remark that the cone corresponding to the Euclidean Jordan algebra 
$\mathbb{R}^{n+1}$ is called the \emph{Lorentz cone}. 
This case was investigated in \cite{Kol10}. Therefore, in what follows we will 
implicitly assume that the symmetric cone $\mathscr{V}$ is not the Lorentz cone.

As we wrote in at the beginning of this paper, our main aim is to present a characterization theorem for the 
Wishart distribution. Therefore, we also have to review some notations concerning it. For further details, we refer to 
Eaton \cite[Chapter VIII., 302--333]{Eat07}. 

 Let $a\in \mathscr{V}$ and
\[
 p\in \left\{0, \frac{d}{2}, \ldots, \frac{d(r-1)}{2}\right\}\cup ]d(r-1), +\infty[. 
\]
We say that the random variable $X$ has \emph{Wishart distribution} with scalar parameter $a$ 
and shape parameter $p$, if the Laplace transform of its distribution is 
\[
 \int_{\overline{\mathscr{V}}}\exp\left(-\langle t, y\rangle\right)d(\gamma_{p, a}y)=
\dfrac{1}{\left(\det\left(\mathbf{1}+ta^{-1}\right)\right)^{p}}, 
\]
where $\Gamma_{\mathscr{V}}$ denotes the multivariate Gamma function. 

If $p\leq \dfrac{\dim(\mathscr{V})}{r}-1$ then the Wishart no longer has a density, but it represents a singular distribution. 
If $p > \dfrac{\dim(\mathscr{V})}{r}-1$ then the Wishart distribution is absolutely continuous with density function
\[
 \gamma_{p, a}(dy)=\dfrac{\det(a)^{p}}{\Gamma_{\mathscr{V}}(p)}\left(\det(y)\right)^{p-\frac{\mathscr{V}}{r}}
\cdot \exp\left(-\langle a, y\rangle\right)I_{\mathscr{V}}(y)dy. 
\]

\subsection{Prerequisites from the theory of functional equations}

This subsection contains some basic definitions and results from the theory of functional 
equations. 
Concerning this topic we refer to the two basic monographs 
Acz\'{e}l \cite{Acz66} and Kuczma \cite{Kuc09}. 

Subsequently we will work on the cone $\mathscr{V}$. Therefore, when reviewing some definition 
(e.g. the notion of additive and logarithmic mappings, resp.) we will always restrict 
our considerations only to a Euclidean Jordan algebra $\mathbb{E}$. 

\begin{defin}
Let $A\subset \mathbb{E}$ be an arbitrary nonempty set and
\[
\mathscr{A}=\left\{(x, y)\in\mathbb{E}^{2}\, \vert\,  x, y, x+y\in A\right\}.
\]
A function $a\colon A\to\mathbb{R}$ is called \emph{additive on $A$}, if for all $(x, y)\in\mathscr{A}$
\begin{equation}\label{Eq1.3.1}
a(x+y)=a(x)+a(y).
\end{equation}
If $A=\mathbb{E}$, then the function $a$ will be called simply \emph{additive}.
\end{defin}

It is well-known that the solutions of the equation above, under some
mild regularity condition, are of the form
\[
 a(x)=\langle\lambda, x\rangle \qquad \left(x\in A\right),
\]
with a certain constant $\lambda\in\mathbb{E}$. 
It is also known, however, that there are additive functions that are nowhere continuous, 
see Kuczma \cite{Kuc09}.

In the sequel we will use the following extension theorem concerning the 
so-called Pexider equation, this result is a special case of 
\cite[Theorem 3]{ChuTab08} if we choose the normed space 
$X$ to be $\mathbb{E}$ and the open and connected set 
$D$ to be $\mathscr{V}\times \mathscr{V}$. 

\begin{thm}
Assume that for the functions $k, l, n\colon \mathscr{V}\to \mathbb{R}$  
\[
 k(x+y)=l(x)+n(y)
\]
is fulfilled for any $x, y\in \mathscr{V}$. 
If the function $k$ is nonconstant, then these functions can be 
\emph{uniquely} extended to functions 
$\widetilde{k}, \widetilde{l}, \widetilde{n}\colon \mathbb{E}\to \mathbb{R}$ so that 
\[
 \widetilde{k}(x+y)=\widetilde{l}(x)+\widetilde{n}(y)  
\qquad 
\left(x, y\in \mathbb{E}\right). 
\]
\end{thm}

Especially, if a 
function $a\colon \mathscr{V}\to \mathbb{R}$ is additive on $\mathscr{V}$, then it can  
always be \emph{uniquely} extended to an additive function $\widetilde{a}\colon\mathbb{E}\to \mathbb{R}$. 
We remark this follows also from 
Theorem 4 of P\'{a}les \cite{Pal02}. 

In what follows we will present the general solution of the Pexider equation on a restricted domain. This 
theorem follows immediately from \cite[Theorem 1]{ChuTab08}, with exactly the same choice as above. 

\begin{thm}
 Assume that for the functions $k, l, n\colon \mathscr{V}\to \mathbb{R}$  
\[
 k(x+y)=l(x)+n(y)
\]
is fulfilled for any $x, y\in \mathscr{V}$ and that 
the function $k$ is nonconstant. 
Then and only then there exists a uniquely determined 
additive function $a\colon \mathbb{E}\to \mathbb{R}$ and real constants 
$b$ and $c$ so that 
\[
 \begin{array}{rcl}
  k(x)&=&a(x)+b+c\\
l(x)&=&a(x)+b\\
n(x)&=&a(x)+c
 \end{array}
\qquad 
\left(x\in \mathscr{V}\right). 
\]
\end{thm}

Finally, the following statement concerns the constant solutions of the Pexider equation.

\begin{cor}
 Let $k\in \mathbb{R}$ be fixed and 
$l, n\colon \mathscr{V}\to \mathbb{R}$ be functions so that 
\[
 k= l(x)+n(y)
\]
is fulfilled for all $x, y\in \mathscr{V}$. 
Then there exists $c\in \mathbb{R}$ such that 
\[
 n(x)= c 
\quad 
\text{and} 
\quad l(x)=k-c 
\qquad
\left(x\in \mathscr{V}\right). 
\]

\end{cor}

The following lemma will also be used frequently in the sequel. 

\begin{lemma}
 Let $a\colon \mathscr{V}\to\mathbb{R}$ be a additive function 
and assume that at least one of the following statements is valid. 
\begin{enumerate}[(i)]
 \item The function $a$ is continuous at a point $x_{0}\in\mathscr{V}$. 
\item There exists a set $A\subset \mathscr{V}$ of positive Lebesgue measure such that 
the function $a$ is bounded above or below on $A$;
\item There exists a set $A\subset \mathscr{V}$ of positive Lebesgue measure such that 
the restriction of $a$ to the set $A$ is measurable (in the sense of Lebesgue). 
\end{enumerate}
Then there exists $\lambda\in\mathbb{E}$ such that 
\[
 a(x)=\langle \lambda, x\rangle  
\qquad 
\left(x\in \mathscr{V}\right). 
\]
\end{lemma}
\begin{proof}
 Due to the results of this section, the additive function 
$a\colon \mathscr{V}\to \mathbb{R}$ can the uniquely extended to an additive function 
defined on $\mathbb{E}$. We will denote this extension also by $a$. 
This means that the additive function $a\colon \mathbb{E}\to \mathbb{R}$ satisfies at least one of 
(i), (ii) or (iii). 
Due the the results of Sz\'{e}kelyhidi \cite{Sze85} this implies that 
$a\colon \mathbb{E}\to \mathbb{R}$ is continuous on $\mathbb{E}$. Therefore, there exists a unique
$\lambda \in \mathbb{E}$ so that 
\[
 a(x)=\langle \lambda, x\rangle 
\]
for all $x\in \mathbb{E}$. 
\end{proof}

A great number of basic functional equations can easily be reduced to \eqref{Eq1.3.1}. 
In the following, we list some of them.

\begin{defin}
Let $A\subset \mathbb{R}$ be a nonempty set and 
\[
\mathscr{M}=\left\{(x,y)\in\mathbb{R}^{2}\, \vert\,  x, y, xy\in A\right\}.
\]
A function $m\colon A\to\mathbb{R}$ is called
\emph{multiplicative on $A$}, if for all $(x, y)\in\mathscr{M}$
\begin{equation*}
m(xy)=m(x)m(y).
\end{equation*}
If $A=\mathbb{R}$ then the function $m$ is called simply \emph{multiplicative}.

Furthermore, we say that the function
$\ell\colon A\to\mathbb{R}$ is \emph{logarithmic on $A$} if for any
$(x,y)\in\mathscr{M}$,
\begin{equation*}
\ell(xy)=\ell(x)+\ell(y)
\end{equation*}
holds. 
\end{defin}

During the proof of our main result, we will use the following result of A.~J\'{a}rai, see 
also \cite{Jar99, Jar05}. 

\begin{thm}[J\'{a}rai]
 Let $Z$ be a regular topological space, 
$Z_{i} \, (i=1, \ldots, n)$ be topological spaces and $T$ be a first countable 
topological space. 
Let $Y$ be an open subset of $\mathbb{R}^{k}$, 
$X_{i}$ an open subset of $\mathbb{R}^{r_{i}}\, (i=1, \ldots, n)$ and 
$D$ an open subset of $T\times Y$. 
Let further $T'\subset T$ be a dense subset, 
$F\colon T'\to Z$, 
$g_{i}\colon D\to X_{i}$ and 
$h\colon D\times Z_{1}\times \cdots\times Z_{n}\to Z$. 
Suppose that the function $f_{i}$ is defined almost everywhere on $X_{i}$ 
(with respect to the $r_{i}$-dimensional Lebesgue measure) with values in $Z_{i}$ 
and the following conditions satisfied. 
\begin{enumerate}[(i)]
 \item for all $t\in T'$ and for almost all $y\in D_{t}=\left\{y\in Y\, (t, y)\in D\right\}$ 
\[
\tag{$\ast$} F(t)=h\left(t, y, f_{1}\left(g_{1}(t, y)\right), \ldots, f_{n}\left(g_{n}(t, y)\right)\right);
\]
\item for all fixed $y\in Y$, the function $h$ is continuous in the other variables;
\item $f_{i}$ is Lebesgue measurable on $X_{i}\, (i=1, \ldots, n)$;
\item $g_{i}$ and also the partial derivative $\dfrac{\partial g_{i}}{\partial y}$ are continuous on $D$ for all 
$i=1, \ldots, n$;
\item for each $t\in T$ there exists $y$ such that $(t, y)\in D$ and the partial derivative 
$\dfrac{\partial g_{i}}{\partial y}$ has rank $r_{i}$ at $(t, y)\in D$ for all $i=1, \ldots, n$. 
\end{enumerate}
Then there exists a unique continuous function $\widetilde{F}$ such that 
$F=\widetilde{F}$ almost everywhere on $T$ and if $F$ is replaced by $\widetilde{F}$ in 
$(\ast)$, then equation $(\ast)$ is satisfied everywhere. 
\end{thm}

\section{The Olkin--Baker functional equation}

In what follows we will investigate the so-called \emph{Olkin--Baker equation} on the cone $\mathscr{V}$, 
that is, functional equation 
\begin{equation}\label{Eq2.1}
 a(x)+b(y)=c(x+y)+d\left(\mathbb{P}\left((x+y)^{-\frac{1}{2}}\right)x\right) 
\qquad 
\left(x, y\in \mathscr{V}\right), 
\end{equation}
where $a, b, c\colon \mathscr{V}\to\mathbb{R}$ and 
$d\colon \mathscr{D}\to \mathbb{R}$ are unknown functions and 
\[
 \mathscr{D}=\left\{z\in\mathscr{V}\, \vert \, \mathbf{1}-z\in\mathscr{V}\right\}. 
\]

Firstly, we will determine the general solutions. After that the a description of the 
regular solutions will follow. 

In 1975, during the Twelfth International Symposium on Functional E\-quations I.~Olkin (see Olkin \cite{Olk75}) posed the problem of 
solving the function equation 
\[
 f(x)g(y)=p(x+y)q\left(\frac{x}{y}\right) 
\qquad 
\left(x, y\in ]0, +\infty[\right), 
\]
where the unknown functions $f, g, p, q\colon ]0, +\infty[\to\mathbb{R}$ are assumed to be positive. 
The general solution of this equation was described in Baker \cite{Bak76}. 
After that, this functional equation was investigated by several authors. 
For example, in Lajk\'{o}--M\'{e}sz\'{a}ros \cite{LajMes11} is was assumed the the functional 
equation is satisfied for almost all pairs $(x, y)\in ]0, +\infty[^{2}$ and the unknown functions 
are measurable. Furthermore, in Ger--Misiewicz--Weso{\l}owski \cite{GerMisWes12} 
it is supposed only that the above equation is fulfilled for almost all pairs $(x, y)\in ]0, +\infty[^{2}$. 
However, no regularity assumption was imposed on the unknown positive functions 
$f, g, p, q$.

\subsection{The general solution}

In this subsection we begin with the description of the general solution of equation 
\eqref{Eq2.1}. Since we assumed that the cone $\mathscr{V}$ is not the Lorentz cone, equation 
\eqref{Eq2.1} has the form 
\begin{equation}\label{Eq2.2}
 a(x)+b(y)=c(x+y)+d\left((x+y)^{-\frac{1}{2}}x(x+y)^{-\frac{1}{2}}\right). 
\end{equation}

\begin{lemma}\label{lemma}
  Let us assume that the functions $a, b, c\colon \mathscr{V}\to\mathbb{R}$ and 
$d\colon \mathscr{D}\to \mathbb{R}$
fulfill equation \eqref{Eq2.2} for all $x, y\in \mathscr{V}$. 
Then there exists an additive function $A\colon \mathscr{V}\to\mathbb{R}$, logarithmic functions 
$\ell_{1}, \ell_{2}\colon ]0, +\infty[\to\mathbb{R}$ and 
$\kappa_{1}, \kappa_{2}\in\mathbb{R}$ such that 
\[
 \begin{array}{rcl}
  a(x)&=& A(x)+\ell_{1}\left(\det(x)\right)+e(x)\\
 b(x)&=& A(x)+\ell_{2}\left(\det(x)\right)+f(x)\\
 c(x)&=& A(x)+(\ell_{1}+\ell_{2})\left(\det(x)\right)+g(x)\\
d(u)&=& \ell_{1}\left(\det(u)\right)+\ell_{2}\left(\det(\mathbf{1}-u)\right)+h(u)
 \end{array}
\]
for all $x\in \mathscr{V}$
where the functions $e, f, g\colon \mathscr{V}\to\mathbb{R}$ and 
$h\colon \mathscr{D}\to \mathbb{R}$ satisfy the Olkin--Baker equation 
\[
  e(x)+f(y)=g(x+y)+h\left((x+y)^{-\frac{1}{2}}x(x+y)^{-\frac{1}{2}}\right).
\]

and $e, f$ and $g$ are homogeneous of zero order, that is 
\[
 e(sx)=e(x), \; f(sx)=f(x), \;  \text{and} \; g(sx)=g(x) 
\]
is fulfilled for arbitrary $s\in ]0, +\infty[$ and $x\in\mathscr{V}$. 
\end{lemma}
\begin{proof}
 Let us assume that for the unknown functions $a, b, c$ and $d$ equation 
\eqref{Eq2.2} is valid. 
Let $s\in ]0, +\infty[$ be arbitrary and substitute $sx$ and $sy$ in place 
of $x$ and $y$, respectively to get 
\[
  a(sx)+b(sy)=c(s(x+y))+d\left((x+y)^{-\frac{1}{2}}x(x+y)^{-\frac{1}{2}}\right) 
\qquad 
\left(x, y\in \mathscr{V}\right). 
\]
This yields that 
\[
 a_{s}(x)+b_{s}(y)=c_{s}(x+y) 
\qquad 
\left(s\in ]0, +\infty[, x, y\in\mathscr{V}\right), 
\]
where the functions $a_{s}, b_{s}, c_{s}\colon\mathscr{V}\to\mathbb{R}$ are 
defined by 
\[
 \begin{array}{rcl}
  a_{s}(x)&=&a(sx)-a(x)\\[1.5mm]
b_{s}(x)&=&b(sx)-b(x)\\[1.5mm]
c_{s}(x)&=&c(sx)-c(x)
 \end{array}
\qquad 
\left(s\in ]0, +\infty[, x\in \mathscr{V}\right). 
\]
Thus there exists an additive function $A_{s}\colon \mathscr{V}\to\mathbb{R}$ and 
functions  $\alpha, \beta \colon \allowbreak ]0, +\infty[\to\mathbb{R}$ such that 
\[
 \begin{array}{rcl}
  a_{s}(x)&=&A_{s}(x)+\alpha(s)\\[1.5mm]
b_{s}(x)&=&A_{s}(x)+\beta(s)\\[1.5mm]
c_{s}(x)&=&A_{s}(x)+\alpha(s)+\beta(s)
 \end{array}
\qquad 
\left(s\in ]0, +\infty[, x\in\mathscr{V}\right). 
\]
Let $s, t\in ]0, +\infty[$ and $z\in \mathscr{V}$, in view of the definition of 
the function $a_{s}$, 
\[
 a_{st}(z)=a((st)z)-a(z)
=a((st)z)-a(sz)+a(sz)-a(z)=
a_{t}(sz)+a_{s}(z). 
\]
Therefore, 
\[
 A_{st}(z)+\alpha(st)=
A_{t}(sz)+\alpha(t)+A_{s}(z)+\alpha(s)
\quad 
\left(s, t\in ]0, +\infty[, z\in \mathscr{V}\right)
\]
holds, that is 
\[
 A_{st}(z)-A_{t}(sz)-A_{s}(z)=
\alpha(t)+\alpha(s)-\alpha(st)
\quad 
\left(s, t\in ]0, +\infty[, z\in \mathscr{V}\right). 
\]
Since the right hand side of this identity does not depend on 
$z\in\mathscr{V}$, however, on the left hand side $z\in\mathscr{V}$ can 
be arbitrary, we obtain that 
\[
 \alpha(st)=\alpha(s)+\alpha(t)
\]
holds for all $s, t\in ]0, +\infty[$. This means that the 
function $\alpha\colon ]0, +\infty[$ is logarithmic. 
A similar computation shows that the function 
$\beta \colon ]0, +\infty[\colon \mathbb{R}$ is also logarithmic on 
$]0, +\infty[$. 

Thus there exist logarithmic functions 
$\ell_{1}, \ell_{2}\colon ]0, + \infty[\to\mathbb{R}$ 
such that 
\[
 \alpha(s)=\ell_{1}(s) 
\quad 
\text{and}
\quad 
\beta(s)=\ell_{2}(s) 
\qquad
\left(s\in ]0, +\infty[\right)
\]
is fulfilled. 
Let again $s, t\in ]0, +\infty[$ and $z\in\mathscr{V}$ be arbitrary. Then 
\[
 a_{st}(z)=a_{t}(sz)+a_{s}(z)
\]
holds. Since the left hand side is symmetric in $s$ and $t$, we also have 
\[
 a_{st}(z)=a_{s}(tz)+a_{t}(z). 
\]
Hence, 
\[
 a_{t}(sz)+a_{s}(z)=a_{s}(tz)+a_{t}(z) 
\qquad 
\left(s, t\in ]0, +\infty[, z\in \mathscr{V}\right). 
\]
This yields that 
\[
 A_{t}(sz)+\alpha(t)+A_{s}(z)+\alpha(s)=
A_{s}(tz)+\alpha(s)+A_{t}(z)+\alpha(t)
\qquad 
\left(s, t\in ]0, +\infty[, z\in \mathscr{V}\right). 
\]
With the substitution $s=2$ we get 
\[
 A_{t}(z)=A_{2}(tz)- A_{2}(z)
\qquad 
\left(t\in ]0, +\infty[\right). 
\]
Thus 
\[
 a_{s}(z)=A_{2}(sz)-A_{2}(z)+\ell_{1}(s)
\qquad 
\left(s\in ]0, +\infty[, z\in \mathscr{V}\right). 
\]
Define the function $e\colon \mathscr{V}\to\mathbb{R}$ by 
\[
 e(x)=a(x)-A_{2}(x)-\frac{1}{r}\ell_{1}\left(\det(x)\right) 
\qquad 
\left(x\in\mathscr{V}\right),  
\]
where $r$ denotes the rank of the cone $\mathscr{V}$. 
The above identities imply that the function 
$e$ is homogeneous of order zero. 
A similar computation shows that for the function $b_{s}$
\[
 b(s)=A_{2}(sz)-A_{2}(z)+\ell_{2}(s)
\qquad 
\left(s\in ]0, +\infty[, z\in \mathscr{V}\right)
\]
holds. 
Therefore let us define the functions $f, g, h\colon \mathscr{V}\to\mathbb{R}$ through 
\[
 f(x)=b(x)-A_{2}(x)-\frac{1}{r}\ell_{2}(\det(x))
\qquad 
\left(x\in\mathscr{V}\right), 
\]

\[
 g(x)=c(x)-A_{2}(x)-\frac{1}{r}(\ell_{1}+\ell_{2})(\det(x))
\qquad 
\left(x\in\mathscr{V}\right)
\]
and
\[
 h(u)=d(u)- \dfrac{1}{r}\ell_{1}\left(\det(u)\right)-\dfrac{1}{r}\ell_{2}\left(\det(\mathbf{1}-u)\right)
\qquad 
\left(u\in\mathscr{D}\right). 
\]

In this case this functions are $f$ and $g$ are homogeneous of order zero. Furthermore, 
$e, f, g$ and $h$ also satisfy equation \eqref{Eq2.2}.

\end{proof}

In what follows, we will investigate the functions $e, f, g$ and $h$ appearing 
in the previous lemma. 

\begin{lemma}\label{L2.2}
 Let $a, b, c\colon \mathscr{V}\to\mathbb{R}$ and 
$d\colon \mathscr{D}\to \mathbb{R}$ be functions and let us assume that 
equation \eqref{Eq2.2} holds for all $x, y\in \mathscr{V}$. 
If the functions $a, b$ and $c$ are homogeneous of order zero,  
then for the function $c\colon \mathscr{V}\to \mathbb{R}$ functional equation 
\[
 c(yxy)=c(x)+2c(y) 
\qquad 
\left(x, y\in \mathscr{V}\right)
\]
is fulfilled. 
\end{lemma}
\begin{proof}
With out the loss of generality 
\[
 a(\mathbf{1})=b(\mathbf{1})=c(\mathbf{1})=d\left(\dfrac{1}{2}\right)=0
\]
can be assumed, otherwise the us consider the functions 
\[
 \hat{a}(x)=a(x)-a(\mathbf{1})
\quad 
\hat{b}(x)=b(x)-b(\mathbf{1})
\quad 
\hat{c}(x)=c(x)-c(\mathbf{1})
\quad 
\hat{d}(x)=d(x)-d\left(\frac{1}{2}\right)
\qquad 
\left(x\in \mathscr{V}\right)
\]

Let us assume that equation \eqref{Eq2.2} holds for all $x, y\in\mathscr{V}$. 
If we interchange $x$ and $y$ in \eqref{Eq2.2}, we obtain that 
\[
 a(y)+b(x)=c(x+y)+d\left((x+y)^{-\frac{1}{2}}y(x+y)^{-\frac{1}{2}}\right) 
\qquad 
\left(x, y\in\mathscr{V}\right). 
\]
However, equation \eqref{Eq2.2} with the substitution $y=x$ 
yields that 
\[
 a(x)+b(x)=c(2x)+d\left(\frac{1}{2}\right) 
\qquad 
\left(x\in\mathscr{V}\right). 
\]
Therefore, 
\[
 2c(x+y)+d\left((x+y)^{-\frac{1}{2}}x(x+y)^{-\frac{1}{2}}\right)+
d\left((x+y)^{-\frac{1}{2}}y(x+y)^{-\frac{1}{2}}\right)
=c(2x)+c(2y)+2d\left(\frac{1}{2}\right), 
\]
or after rearranging, 
\[
 -2c(x+y)+c(x)+c(y)=
d\left((x+y)^{-\frac{1}{2}}x(x+y)^{-\frac{1}{2}}\right)+
d\left((x+y)^{-\frac{1}{2}}y(x+y)^{-\frac{1}{2}}\right)
-2d\left(\frac{1}{2}\right)
\]
holds for all $x, y\in\mathscr{V}$, where the zero order homogeneity 
of the functions $c$ and $d$ was also used. 
Let $u, v\in\mathscr{V}$ be arbitrary and let us substitute 
\[
 x=(u+v)^{-\frac{1}{2}}u(u+v)^{-\frac{1}{2}} 
\quad 
y=(u+v)^{-\frac{1}{2}}v(u+v)^{-\frac{1}{2}}
\]
into the previous identity to receive
\begin{multline*}
 -2c(1)+c\left((u+v)^{-\frac{1}{2}}u(u+v)^{-\frac{1}{2}}\right)+
c\left((u+v)^{-\frac{1}{2}}v(u+v)^{-\frac{1}{2}} \right)
\\=
d\left((u+v)^{-\frac{1}{2}}u(u+v)^{-\frac{1}{2}}\right)+
d\left((u+v)^{-\frac{1}{2}}v(u+v)^{-\frac{1}{2}} \right)+2d\left(\frac{1}{2}\right). 
\end{multline*}
On the other hand, for all $u, v\in\mathscr{V}$
\[
 -2c(u+v)+c(u)+c(v)=
d\left((u+v)^{-\frac{1}{2}}u(u+v)^{-\frac{1}{2}}\right)+
d\left((u+v)^{-\frac{1}{2}}v(u+v)^{-\frac{1}{2}}\right)
-2d\left(\frac{1}{2}\right). 
\]
Thus, 
\begin{multline*}
 -2c(1)+c\left((u+v)^{-\frac{1}{2}}u(u+v)^{-\frac{1}{2}}\right)+
c\left((u+v)^{-\frac{1}{2}}v(u+v)^{-\frac{1}{2}} \right)
\\
=
-2c(u+v)+c(u)+c(v)
\qquad 
\left(u, v\in\mathscr{V}\right). 
\end{multline*}
Let us define the functions $C$ and $D$ on $\mathscr{V}^{2}$ through 
\[
 C(x, y)=-c(x+y)+c(x)+c(y) 
\qquad 
\left(x, y\in\mathscr{V}\right)
\]
and 
\begin{multline*}
 D(x, y)=
-2c(1)+c(x+y)
\\
+c\left((x+y)^{-\frac{1}{2}}x(x+y)^{-\frac{1}{2}}\right)
+c\left((x+y)^{-\frac{1}{2}}y(x+y)^{-\frac{1}{2}}\right)
\quad 
\left(x, y\in \mathscr{V}\right). 
\end{multline*}
In this case the previous equation yields that 
\[
 C(x, y)=D(x, y)
\]
is fulfilled for arbitrary $x, y\in\mathscr{V}$. 
Since the function $C$ is a Cauchy difference, it satisfies the 
cocycle equation, i.e., 
\[
 C(x+y, z)+C(x, y)=C(x, y+z)+C(y, z) 
\qquad 
\left(x, y\in\mathscr{V}\right), 
\]
furthermore, the definition of the function $C$ 
immediately implies that $C$ is a symmetric function. 
Therefore, the function $D$ is also symmetric and 
\[
 D(x+y, z)+D(x, y)=D(x, y+z)+D(y, z) 
\]
holds for all $x, y, z\in \mathscr{V}$. This equation 
with $y=z$ yields that 
\[
 D(x+y, z)+D(x, z)=D(x, 2z)+D(z, z)
\]
or 
\[
 D(x, z)=D(x, 2z)+D(z, z)-D(x+z, z) 
\qquad 
\left(x, z\in\mathbb{V}\right). 
\]
Now using the definition of the function $D$, we get that 
\begin{multline*}
 c(x+z)
+c\left((x+z)^{-\frac{1}{2}}x(x+z)^{-\frac{1}{2}}\right)
+c\left((x+z)^{-\frac{1}{2}}x(x+z)^{-\frac{1}{2}}\right)
\\
=
c(x+2z)
+c\left((x+2z)^{-\frac{1}{2}}x(x+2z)^{-\frac{1}{2}}\right)
+c\left((x+2z)^{-\frac{1}{2}}2z(x+2z)^{-\frac{1}{2}}\right)
\\
+c(z)
-c(x+2z)-c\left((x+2z)^{-\frac{1}{2}}(x+z)(x+2z)^{-\frac{1}{2}}\right)
\\
-c\left((x+2z)^{-\frac{1}{2}}z(x+2z)^{-\frac{1}{2}}\right), 
\end{multline*}
or after some rearrangement, 
\begin{multline*}
 c(x+z)
+c\left((x+z)^{-\frac{1}{2}}x(x+z)^{-\frac{1}{2}}\right)
+c\left((x+z)^{-\frac{1}{2}}x(x+z)^{-\frac{1}{2}}\right)
\\
=
c\left((x+2z)^{-\frac{1}{2}}x(x+2z)^{-\frac{1}{2}}\right)
+c(z)
\\
-c\left((x+2z)^{-\frac{1}{2}}(x+z)(x+2z)^{-\frac{1}{2}}\right)
\\
\left(x, z\in\mathscr{V}\right). 
\end{multline*}
On the other hand, $D(x, z)=C(x, z)$, that is, 
\[
 -c(x+z)+c(x)+c(z)
=
c(x+z)
+c\left((x+z)^{-\frac{1}{2}}x(x+z)^{-\frac{1}{2}}\right)
+c\left((x+z)^{-\frac{1}{2}}x(x+z)^{-\frac{1}{2}}\right). 
\]
Thus, for all $x, z\in\mathscr{V}$ we have 
\begin{multline*}
 -c(x+z)+c(x)+c(z)
\\
=
c\left((x+2z)^{-\frac{1}{2}}x(x+2z)^{-\frac{1}{2}}\right)
+c(z)
-c\left((x+2z)^{-\frac{1}{2}}(x+z)(x+2z)^{-\frac{1}{2}}\right), 
\end{multline*}
that is, 
\begin{multline*}
 -c(x+z)+c(x)
\\
=
c\left((x+2z)^{-\frac{1}{2}}x(x+2z)^{-\frac{1}{2}}\right)
-c\left((x+2z)^{-\frac{1}{2}}(x+z)(x+2z)^{-\frac{1}{2}}\right)
\\
\left(x, z\in\mathscr{V}\right). 
\end{multline*}

Furthermore, for all $x, z\in\mathscr{V}$
\begin{multline*}
 -2c(x+2z)+c(x)+c(x+z)
\\
=
c\left((x+2z)^{-\frac{1}{2}}x(x+2z)^{-\frac{1}{2}}\right)
+c\left((x+2z)^{-\frac{1}{2}}(x+z)(x+2z)^{-\frac{1}{2}}\right)
\end{multline*}
Thus, 
\[
 c(x)-c(x+2z)=c\left((x+2z)^{-\frac{1}{2}}x(x+2z)^{-\frac{1}{2}}\right)
\quad 
\left(x, z\in\mathscr{V}\right)
\]
holds. 
Let us observe that this latter identity yields that 
\begin{equation}\label{Eq2.3}
 c(x)=c\left(y^{-\frac{1}{2}}xy^{-\frac{1}{2}}\right)+c(y)
\qquad 
\left(x, y\in\mathscr{V}\right). 
\end{equation}

\end{proof}

\begin{rem}
 The case $\mathscr{V}=]0, +\infty[$ is trivial, since if a function $f\colon ]0, +\infty[\allowbreak\to\mathbb{R}$ is 
homogeneous of order zero, i.e., 
\[
 f(sx)=f(x) 
\qquad 
\left(s, x\in ]0, +\infty[\right), 
\]
then with the substitution $s=\dfrac{1}{x}$ we obtain that 
\[
 f(x)=f(1) 
\qquad 
\left(x\in ]0, +\infty[\right), 
\]
which means that $f$ is a constant function. Therefore, in case $\mathscr{V}=]0, +\infty[$, the general solution of 
equation \eqref{Eq2.2} is 
\[
 \begin{array}{rcl}
  a(x)&=& A(x)+\ell_{1}\left(\det(x)\right)-\kappa\\
 b(x)&=& A(x)+\ell_{2}\left(\det(x)\right)+\kappa\\
 c(x)&=& A(x)+\left(\ell_{1}+\ell_{2}\right)\left(\det(x)\right)\\
d(u)&=& \ell_{1}\left(\det(u)\right)+\ell_{2}\left(\det(\mathbf{1}-u)\right)
 \end{array}
\quad 
\left(x\in \mathscr{V}, u\in \mathscr{D}\right), 
\]
where 
$A\colon \mathscr{V}\to\mathbb{R}$ is an additive function,  
$\ell_{1}, \ell_{2}\colon ]0, +\infty[\to\mathbb{R}$ are logarithmic functions on 
$]0, +\infty[$ 
and $\kappa\in\mathbb{R}$ is a certain constant. 
\end{rem}

\subsection{Regular solutions}

Using the results of the previous subsection, we are able to determine the 
regular solutions of equation \eqref{Eq2.2}. This result is contained in the 
following statement. 

\begin{thm}\label{T2.2}
  Let $a, b, c\colon \mathscr{V}\to\mathbb{R}$ and 
$d\colon \mathscr{D}\to \mathbb{R}$ be functions and let us assume that 
\eqref{Eq2.2} holds for all $x, y\in \mathscr{V}$. If the functions 
$a, b, c$ and $d$ are continuous, then 
\[
 \begin{array}{rcl}
  a(x)&=& A(x)+C_{1}\ln\left(\det(x)\right)-\kappa\\
 b(x)&=& A(x)+C_{2}\ln\left(\det(x)\right)+\kappa\\
 c(x)&=& A(x)+(C_{1}+C_{2})\ln\left(\det(x)\right)\\
d(u)&=& C_{1}\ln\left(\det(u)\right)+C_{2}\ln\left(\det(\mathbf{1}-u)\right)
 \end{array}
\]
holds for all $x\in \mathscr{V}$
and $u\in\mathscr{D}$, where the function $A\colon \mathscr{V}\to\mathbb{R}$ is a continuous additive function. 
\end{thm}

\section{The main result}

In view of the results of the previous section, we are able to prove our main theorem. 

\begin{thm}
 Let $X$ and $Y$ be independent random variables valued in a symmetric cone $\mathscr{V}$ of rank $r> 2$, 
with strictly positive and Lebesgue measurable densities. 
Let further 
\[
 V=X+Y \qquad \text{and} \qquad 
U=(X+Y)^{-\frac{1}{2}}X(X+Y)^{-\frac{1}{2}}. 
\]
If $U$ and $V$ are independent then there exist $a\in\mathscr{V}$ and $p_{1}, p_{2}>\frac{\dim(\mathscr{V})}{r}-1$ such that 
\[
 X\sim \gamma_{p_{1}, a} \quad 
\text{and}
\quad 
Y\sim \gamma_{p_{2}, a}
\]
holds. 
\end{thm}
\begin{proof}
 Let $f_{X}, f_{Y}, f_{U}$ and $f_{V}$ denote the densities of $X, Y, U$ and $V$, respectively. 
Furthermore, let us consider the functions $a, b, c\colon \mathscr{V}\to \mathbb{R}$ and $d\colon\mathscr{D}\to \mathbb{R}$ defined by 
\[
\begin{array}{rcl}
 a(x)&=& \ln\left(f_{X}(x)\right),  
\\[3mm]
b(x)&=&\ln\left(f_{Y}(x)\right),  
\\[3mm]
c(x)&=&\ln\left(f_{V}(x)\right)-\dfrac{\dim(\mathscr{V})}{r}\ln\left(\det(x)\right),  
\\[3mm]
d(u)&=&\ln\left(f_{V}(u)\right)
\end{array}
\quad 
\left(x\in \mathscr{V}, u\in\mathscr{D}\right). 
\]
In this case 
\[
 a(x)+b(y)=c(x+y)+d\left(\mathbb{P}\left((x+y)^{-\frac{1}{2}}\right)x\right)
\]
holds for almost all $x, y\in \mathscr{V}$, where the functions $a, b, c$ and $d$ are measurable. 
Now, applying the theorem of J\'{a}rai successively, 
in a similar way as in Lemma 1 of M\'{e}sz\'{a}ros \cite{Mes10} we get the following. 
There exist uniquely determined, 
continuous functions $\widetilde{a}, \widetilde{b}, \widetilde{c}\colon \mathscr{V}\to \mathbb{R}$ and $\widetilde{d}\colon \mathscr{D}\to \mathbb{R}$ 
such that 
\[
 a=\widetilde{a}, \;
b=\widetilde{b}, \;
c=\widetilde{c}, \quad 
\text{and}
\quad 
d=\widetilde{d}
\]
holds almost everywhere and 
\[
a(x)+b(y)=c(x+y)+d\left(\mathbb{P}\left((x+y)^{-\frac{1}{2}}\right)x\right) 
\]
is valid for all $x, y\in \mathscr{V}$. 
In view of Theorem \ref{T2.2}, this means that 
\[
 \begin{array}{rcl}
  a(x)&=&\langle \lambda, x\rangle +\ell_{1}\left(\det(x)\right)-\kappa\\[2mm]
 b(x)&=& \langle \lambda, x\rangle+\ell_{2}\left(\det(x)\right)+\kappa
 \end{array}
\qquad 
\left(x\in \mathscr{V}\right). 
\]
Thus 
\[
 \begin{array}{rcl}
 f_{X}(x)&=&\exp(a(x))= \exp(C_{1})\exp(\langle \lambda, x\rangle)\det(x)^{k_{1}}\\[2mm]
 f_{Y}(x)&=&\exp(b(x))=\exp(C_{2})\exp(\langle \lambda, x\rangle)\det(x)^{k_{2}}
 \end{array}
\qquad 
\left(x\in \mathscr{V}\right). 
\]
Since, $f_{X}$ and $f_{Y}$ are densities, we obtain that $a=\lambda\in \mathscr{V}$, 
\[
 k_{1}=p_{1}-\dfrac{\dim(\mathscr{V})}{r}>-1 
\qquad 
k_{2}=p_{2}-\dfrac{\dim(\mathscr{V})}{r}>-1 
\]
and 
\[
 \exp(C_{1})=\dfrac{\det(a)^{p_{1}}}{\Gamma_{\nu}(p_{1})} 
\qquad 
\exp(C_{1})=\dfrac{\det(a)^{p_{1}}}{\Gamma_{\nu}(p_{1})}. 
\]
All in all, 
\[
 X\sim \gamma_{p_{1}, a} \quad 
\text{and}
\quad 
Y\sim \gamma_{p_{2}, a}
\]
can be concluded. 
\end{proof}

\subsection*{Acknowledgement}
I am indebted to Professors Lajos Moln\'{a}r and M\'{a}ty\'{a}s Barczy for 
drawing my attention to this problem. The author gratefully acknowledges the many helpful suggestions of the anonymous referee. 

\providecommand{\bysame}{\leavevmode\hbox to3em{\hrulefill}\thinspace}
\providecommand{\MR}{\relax\ifhmode\unskip\space\fi MR }
\providecommand{\MRhref}[2]{%
  \href{http://www.ams.org/mathscinet-getitem?mr=#1}{#2}
}
\providecommand{\href}[2]{#2}

\begin{flushright}
 MTA--DE 'Lend\"{u}let'\\
Functional Analysis Research Group\\
Institute of Mathematics\\
University of Debrecen\\
P. O. Box: 12.\\
Debrecen\\
H--4010\\
Hungary
\end{flushright}
\end{document}